\documentclass[11pt,twoside]{amsart}
\usepackage{amssymb, amsmath, mathrsfs, mathtools, amsfonts, amsthm, amscd, yfonts}
\usepackage{xcolor}
\usepackage{hyperref} 
\usepackage{multicol}

\theoremstyle{plain}
\newtheorem{Theorem}{Theorem}
\newtheorem{Proposition}[Theorem]{Proposition}
\newtheorem{Lemma}[Theorem]{Lemma}

\newtheorem{Conjecture}[Theorem]{Conjecture}
\theoremstyle{definition}
\newtheorem{Definition}[Theorem]{Definition}
\theoremstyle{remark}
\newtheorem{Remark}[Theorem]{Remark}
\newtheorem{Example}[Theorem]{Example}

\begin{document}

\title[Regular algebras in positive characteristic]{A negative answer to a Bahturin-Regev Conjecture about regular algebras in positive characteristic}

\author[L. Centrone]{Lucio Centrone}
\address{Dipartimento di Matematica, Universit\`a degli Studi di Bari Aldo Moro, Via Edoardo Orabona, 4, 70125 Bari, Italy}
\email{lucio.centrone@uniba.it}
\thanks{L. Centrone was partially supported by PNRR-MUR PE0000023-NQSTI}
\author[P. Koshlukov]{Plamen Koshlukov}
\address{IMECC, UNICAMP, Rua S\'ergio Buarque de Holanda 651, 13083-859 Campinas, SP, Brazil}
\email{plamen@unicamp.br}
\thanks{P. Koshlukov was partially supported by FAPESP Grant 2018/23690-6 and by CNPq Grant 307184/2023-4}
\author[K. Pereira]{Kau\^e Pereira}
\address{IMECC, UNICAMP, Rua S\'ergio Buarque de Holanda 651, 13083-859 Campinas, SP, Brazil}
\email{k200608@dac.unicamp.br}
\thanks{K. Pereira was supported by FAPESP Grant 2023/01673-0}



\begin{abstract}
Let $A=A_1\oplus\cdots\oplus A_r$ be a decomposition of the algebra $A$ as a direct sum of vector subspaces. If for every choice of the indices $1\le i_j\le r$ there exist $a_{i_j}\in A_{i_j}$ such that the product $a_{i_1}\cdots a_{i_n}\ne 0$, and for every $1\le i,j\le r$ there is a constant $\beta(i,j)\ne 0$ with $a_ia_j=\beta(i,j) a_ja_i$ for $a_i\in A_i$, $a_j\in A_j$, the above decomposition is \textsl{regular}. 
Bahturin and Regev raised the following conjecture: suppose the regular decomposition comes from a group grading on $A$, and form the $r\times r$ matrix whose $(i,j)$th entry equals $\beta(i,j)$. Then this matrix is invertible if and only if the decomposition is minimal (that is one cannot get a regular decomposition of $A$ by coarsening the decomposition). Aljadeff and David proved that the conjecture is true in the case the base field is of characteristic 0. We prove that the conjecture does not hold for algebras over fields of positive characteristic, by constructing algebras with minimal regular decompositions such that the associated matrix is singular. 
\end{abstract}
\maketitle
\section{Introduction}
Group gradings on algebras appeared naturally in mathematics: the usual polynomial algebra in one or several variables over a field has a natural grading (the degree) by the infinite cyclic group $\mathbb{Z}$. Group gradings also appear naturally in the study of Lie and Jordan superalgebras. Wall in \cite{wall} described the central $\mathbb{Z}_2$-graded simple associative algebras over an algebraically closed field. 
The key role of gradings in PI theory was highlighted by A. Kemer in 1984-1986. In fact Kemer developed a sophisticated theory in order to classify the ideals of identities of associative algebras in characteristic 0 and this led him to the positive solution of the Specht problem. We refer the reader to \cite{kemer} for further details. Soon afterwards group gradings on algebras and the corresponding graded polynomial identities became a widespread trend in PI theory. We direct the interested readers to the monograph \cite{eldkochetov} and the references therein for the state-of-art in the area. A grading on an algebra means one has an additional structure, this implies that more information can be deduced by using such an additional structure. Let us recall that the polynomial identities for the matrix algebras $M_n(K)$ are known for $n\le 2$ only, if $K$ is an infinite field of characteristic not 2, see \cite{razmm2}, \cite{drenskym2}, \cite{pkm2}. On the other hand, the graded identities for $M_n(K)$, under the natural $\mathbb{Z}_n$-grading are known for every $n$, and every infinite field \cite{vasmn}, \cite{ssamn}. If two algebras satisfy the same graded identities they are PI equivalent (that is they satisfy the same ordinary identities). It should be noted that the knowledge of the graded identities satisfied by a given algebra does not yield much information about the ordinary ones. 

The notion of a regular algebra was introduced by Regev and Seeman in \cite{regevz2}. Assume $A$ is an algebra over $K$, and let $G$ be a finite group. A $G$-grading on $A$ is a decomposition $A=\oplus_{g\in G} A_g$, a direct sum of vector subspaces, such that $A_gA_h\subseteq A_{gh}$ for every $g$, $h\in G$. Such a decomposition of $A$ is called regular if for every $n$-tuple of elements of $G$, $(g_1, \ldots, g_n)$, there exist elements $a_i\in A_{g_i}$ such that $a_1\cdots a_n\ne 0$, and for every $g$, $h\in G$ there exists a constant $\beta(g,h)\in K^*$ such that $a_ga_h=\beta(g,h) a_ha_g$, for every $a_g\in A_g$ and $a_h\in A_h$. This is not the most general notion of a regular decomposition (we give it below) but for motivational purposes it suffices. One can form the matrix $M^A=(\beta(g,h))$ of order $|G|$ that is called \textit{decomposition matrix} of $A$. In \cite{regevz2} the authors proved that the matrix $M^A$ determines the multilinear identities of the algebra $A$; as a consequence, in characteristic 0 it determines the polynomial identitites of $A$. The authors of \cite{regevz2} were mostly concerned with the twisted (or $\mathbb{Z}_2$-graded) tensor product of $\mathbb{Z}_2$-graded algebras. They conjectured that the twisted tensor product of two T-prime algebras is PI equivalent to a T-prime algebra, and proved this conjecture in several important cases. The conjecture was proved to be true in its generality, see \cite{dvn, jfpk}. In the latter paper the authors additionally proved that this conjecture does not hold if the base field is infinite and of positive characteristic $p>2$. In \cite{bahturin2009graded, bcommutation} the authors studied the so-called minimal decompositions of $A$: roughly speaking these are regular decompositions that cannot be "coarsened", see the precise definition below. In \cite{bcommutation} the authors computed the determinant of the decomposition matrix for $M_n(F)$ when the latter algebra is graded by the group $\mathbb{Z}_n\times\mathbb{Z}_n$ in the natural way. This determinant equals $\pm n^{n^2}$. Furthermore it was conjectured in \cite{bahturin2009graded} that the decomposition matrix is invertible if and only if the regular decomposition is minimal. Moreover in \cite{bahturin2009graded} it was shown that this is the case in several important situations. Another conjecture was proposed in \cite{bahturin2009graded}: namely that the number of summands in a minimal decomposition of a regular algebra, as well as the determinant of the decomposition matrix, are invariants of the algebra, but do not depend on the decomposition. In \cite{Eli1} both conjectures were solved in the affirmative. It was shown that the the order of the group defining a minimal regular decomposition of a finite dimensional PI algebra equals the PI exponent of the algebra. Recall the latter is one of the most important numerical invariant of a PI algebra, and has been extensively studied since around 1980. Furthermore in \cite{Eli1} the authors proved that if $A$ is $G$-graded, regular and minimal then $\det(M^{A}) = \pm|G|^{|G|/2}$. Additionally, in \cite{Eli1} the authors introduced the notion of a minimal regular grading to the case of a nonabelian group $G$ (they call these nondegenerate gradings), and proved the analogue of the above conjecture in this setting as well. 

The above results were obtained assuming the base field is a large field of characteristic 0: either algebraically closed or at least containing many roots of unity. In the present paper we study regular decompositions over fields of positive characteristic. Our main goal is the conjecture raised by Bahtirn and Regev \cite{bahturin2009graded}, but when the algebras are over a field of positive characteristic. We consider decompositions defined by finite abelian groups. Our main results consist in proving that the conjecture fails in positive characteristic. We provide an explicit example of a $\mathbb{Z}_n\times \mathbb{Z}_n$-graded algebra admitting a minimal regular decomposition and whose decomposition matrix is singular. We have to consider separately the cases of a finite and of an infinite field $K$. In order to make the "picture" complete we provide another example where the corresponding determinant is nonzero. 

\section{Preliminaries}
\subsection{Regular algebras}

In this section we will introduce the main definitions and results used in the sequel. 

Let $K$ be a field which will further be assumed to be of characteristic $p>2$, and $X=\{x_{1},\ldots, x_{n},\ldots\}$ be an infinite countable set of independent noncommuting variables.  Also, by a $K$-algebra $A$, we mean a vector space $A$ over $K$ equipped with a $K$-linear product $A\times A\rightarrow A$, denoted by $\cdot$, or simply by a juxtaposition, that is distributive, associative and has an element $1\in A$ (called {\it unit}) such that $a1=1a=a$, for every $a\in A$.

We start with the following definition (see Definition 2.3 of \cite{regevz2}).

\begin{Definition}
\label{def.reg.0}
Let $A$ be a $K$-algebra admitting a decomposition into a direct sum of vector subspaces in the following form:
\begin{equation}
\label{eq.2.0}
A=A_{1}\oplus A_{2}\oplus \cdots \oplus A_{r}. 
\end{equation}
We say the algebra $A$ is \textsl{regular} (and the decomposition (\ref{eq.2.0}) regular) if it satisfies the following conditions:  
\begin{enumerate}
    \item[(i)] Given $n$ indexes $1\leq i_{1},\ldots, i_{n}\leq r$, there exist elements $a_{1}\in A_{i_{1}}$, \dots, $a_{n}\in A_{i_{n}}$
such that $a_{1}\cdots a_{n}\neq 0$.
    \item[(ii)] Given $1\leq i,j\leq r$, there exists $\beta(i,j)\in K^{\ast}$ (where $K^{\ast}=K\setminus\{0\}$) such that for every $a_{i}\in A_{i}$ and $a_{j}\in A_{j}$ one has
    \[
   a_{i}a_{j}=\beta(i,j)a_{j}a_{i}. 
    \]
\end{enumerate}
\end{Definition}
We draw the readers' attention to the fact $\beta(i,j)$ depends only on $i$ and $j$ but \textsl{not} on the choice of the elements $a_i$ and $a_j$.

\vspace{0.5cm } 

 If $A$ is a regular algebra with a regular decomposition as in (\ref{eq.2.0}), then we define the decomposition matrix associated with the regular decomposition (\ref{eq.2.0}) as
 \[
 M^{A}:=(\beta(i,j))_{i,j}. 
 \] 
 Throughout this paper, we work with regular algebras whose regular decomposition is given by a $G$-grading, where $G$ is a finite abelian group. More precisely, and for future reference, we have the following definition:
 \begin{Definition} 
\label{def.reg.1}
 Let $G$ be a finite abelian group and $A=\oplus_{g\in G} A_{g}$ be a $G$-graded algebra. We say that $A$ is $G$-graded regular (with the regular decomposition given by the $G$-grading) if it satisfies the following conditions:
 \begin{enumerate}
     \item[(i)] Given $n\in \mathbb{N}$ and any $n$-uple $(g_{1},\ldots, g_{n})\in G^{n}$, there exist homogeneous elements $a_{1}\in A_{g_{1}}$, \ldots, $a_{n}\in A_{g_{n}}$ such that $a_{1}\cdots a_{n}\neq 0$.
     \item[(ii)]  For every $g$, $h\in G$, and for every $a_{g}\in A_{g}$ and $a_{h}\in A_{h}$, there exists $\beta(g,h)\in K^{\ast}$ satisfying
     \[
     a_{g}a_{h}=\beta(g,h)a_{h}a_{g}.
     \]
 \end{enumerate}
 Moreover, we define the regular decomposition matrix associated with the regular decomposition of the $G$-graded regular algebra $A$ as $M^{A}:=(\beta(g,h))_{g,h}$.
 \end{Definition}
 
In Definition \ref{def.reg.1}, the algebra $A$ is associative, therefore the function $\beta\colon G\times G\rightarrow K^{\ast}$ satisfies the following conditions.  For every $g$, $h$, $k\in G$ one has
\begin{equation}
\label{eq.2.3}
\beta(g,h)=(\beta(h,g))^{-1}
\end{equation}
\begin{equation}
    \label{eq.2.4}
   \beta(g,h+k)=\beta(g,h)\beta(g,k) 
\end{equation}
\begin{equation}
    \label{eq.2.5}
    \beta(g+k,h)=\beta(g,h)\beta(k,h)
\end{equation}
\begin{Definition}
 A function $\beta\colon G\times G\rightarrow K^{\ast}$  satisfying (\ref{eq.2.3}), (\ref{eq.2.4}), and (\ref{eq.2.5}), is called a bicharacter of $G$.     
\end{Definition}

\begin{Definition}If $A$ is a $G$-graded regular algebra, the map $\beta$ satisfying the second condition of Definition \ref{def.reg.1} will be called bicharacter associated to the regular decomposition of $A$. 
\end{Definition}

Next we will provide some classical examples of regular algebras whose regular decompositions are given by abelian groups. 

From now on, we will use the additive notation for abelian groups. 

\begin{Example} If $A=K[x]$ and $G=\mathbb{Z}_{n}$, we define $A_{i}:=x^{i}K[x^{n}]$, $0\le i\leq n-1$. It follows that $A$ is regular and $G$-graded. Observe that the grading here is the natural grading given by the degrees, modulo $n$, of the polynomials. The regular decomposition is given by $A=\oplus_{i=0}^{n-1} A_{i}$. The bicharacter $\beta$ in this case is the trivial bicharacter ($\beta(i,j)=1$ for every $i$, $j\in\mathbb{Z}_n$), because the algebra $A$ is commutative. 
\end{Example}

\begin{Example} Let $E$ be the infinite dimensional Grassmann algebra on a vector space $V$ with a basis $e_1$, $e_2$, \dots; a linear basis for $E$ is given by 
\[\mathcal{B}=
\{e_{i_{i}}e_{i_{2}}\cdots e_{i_{n}}\mid  i_{1}<i_{2}<\cdots<i_{n},\quad n\in \mathbb{N}\}\cup \{1\}.
\]
Let $E_0$ be the span of the elements of $\mathcal{B}$ with even length and let $E_1$ be the span of the elements of $\mathcal{B}$ with odd length. Then $E_0$ is the centre of $E$ whereas the elements of $E_1$ anticommute.

 It is well known, and easy to check, that the decomposition $E=E_{0}\oplus E_{1}$ turns $E$ into a $\mathbb{Z}_{2}$-graded algebra. Because of the relation 
 \[
 (e_{i_{1}}\cdots e_{i_{n}})(e_{j_{1}}\cdots e_{j_{m}})=(-1)^{nm} (e_{j_{1}}\cdots e_{j_{m}})(e_{i_{1}}\cdots e_{i_{n}})
 \]
 it follows that $E$ is a $\mathbb{Z}_{2}$-graded regular algebra, and the corresponding bicharacter is given by
 \[
 \beta(0,0)=\beta(0,1)=\beta(1,0)=1,\quad \beta(1,1)=-1.
 \]
 In particular, the decomposition matrix is given by $M^{E}=\begin{pmatrix}
     1 & 1  \\
    1  & -1 
     \end{pmatrix}$.
\end{Example}
\begin{Example}
\label{pauli.matrix}
    Given $n\in \mathbb{Z}$, let $\xi\in K$ be a primitive $n$-th root of unity.  We can consider on $M_{n}(K)$, the $\xi$-grading, which is defined in the following way. First we put $X:=diag(\xi^{n-1},\xi^{n-2},\dots,\xi,1)$ and $Y=e_{n,1}+\sum_{i=1}^{n-1} e_{i,i+1}$. Then we set
        \[
        A_{ij}:= \text{\rm span}_{K}\{X^{i} Y^{j}\},\quad \text{for}\quad 0\leq i,j\leq n-1.
        \]
    Since $XY=\xi YX$ one can easily verify that $M_{n}(K)=\oplus_{i,j=0}^{n-1}A_{i,j}$ is a $(\mathbb{Z}_{n}\times \mathbb{Z}_{n})$-grading on $M_{n}(K)$. Moreover, as $X$ and $Y$ are invertible matrices, the products of $X$ and $Y$ are also invertible. The relation  $XY=\xi YX$ gives that $M_{n}(K)$, with the $\xi$-grading, is a regular algebra with bicharacter given by $\beta(i,j)=\xi^{jk-il}$.
\end{Example}
The next example, which at first might seem somewhat artificial, is very important, and plays a fundamental role in this paper.

\begin{Example} 
\label{twisted.group}
Let $G$ be a finite abelian group and let $\alpha\in H^{2}(G,K^{\ast})$, where $H^{2}(G,K^{\ast})$ denotes the second cohomology group of $G$ with values in $K^*$. Recall that $\alpha$ satisfies the {\bf cocycle condition}
 \[
 \alpha(\sigma,\tau+\nu)\alpha(\tau,\nu)=\alpha(\sigma+\tau,\nu)\alpha(\sigma,\tau),\quad \text{for every}\quad \sigma,\tau,\nu\in G.
 \]
The twisted group algebra with respect to $G$ and $\alpha$, denoted by $K^{\alpha}G$, is the algebra defined on  the vector space with basis $\mathcal{C}=\{x_{g}\mid g\in G\}$ and product defined on the basic elements of $\mathcal{C}$ by
 \[
x_{g}x_{h}:= \alpha(g,h)x_{g+ h}\quad \text{for all}\quad g,h\in G. 
 \]  
 The cocycle condition is in fact equivalent to the associativity of the algebra. 
  It is easy to see that $K^{\alpha}G$ is a $G$-graded regular algebra with bicharacter given by
 \[
 \beta(g,h)=\alpha(g,h)\alpha(h,g)^{-1}\quad \text{for every}\quad g,h\in G.
 \]
 This bicharacter is said to be \textsl{induced} by $\alpha$. 
 \end{Example}

\begin{Example} Let $A=A_0\oplus A_1$, where $A=K[x]$ be the polynomial algebra in one variable $\mathbb{Z}_{2}$-graded by the usual degree. More precisely, let $A_0$ be the span of the monomials of even degree and $A_1$ the span of the monomials of odd degree; denote by $\partial(a)$ the homogeneous degree of the homogeneous element $a\in A$. The vector space $A\otimes_K A$ equipped with the product defined on its homogeneous elements by
\[
(a_{1}\otimes b_{1})(a_{2}\otimes b_{2})=(-1)^{\partial(b_{1})\partial(a_{2})}(a_{1}a_{2}\otimes b_{1}b_{2})
\]
is a $\mathbb{Z}_{2}$-graded algebra, too. We denote this new algebra by $A\overline{\otimes} A$ and we call it \textsl{$\mathbb{Z}_{2}$-graded (or twisted) tensor product} of $A$ by $A$. The grading on $A\overline{ \otimes } A$ is given by
\[
A\overline{\otimes } A= C_{0}\oplus  C_{1},
\]
where $C_{0}:= (A_{0}\otimes A_{0})\oplus (A_{1}\otimes A_{1})$, and $ C_{1}:= (A_{0}\otimes A_{1})\oplus (A_{1}\otimes A_{0})$. With respect to this $\mathbb{Z}_{2}$-grading, the twisted tensor product $A\overline{ \otimes } A$ is not regular. 

Indeed, if we take homogeneous basis elements $a_{0}$, $a_{0}', b_0\in A_{0}$, $b_{1}$, $b_{1}'$, $a_{1}\in A_{1}$, so that their product is still an element of the basis, then
\[
(a_{0}\otimes b_{0}+ a_{1}\otimes b_{1})(a_{0}'\otimes b_{1}')=a_{0}a_{0}'\otimes b_{0}b_{1}'+ a_{1}a_{0}'\otimes b_{1}b_{1}'.
\]
On the other hand
\[
(a_{0}'\otimes b_{1}')(a_{0}\otimes b_{0}+ a_{1}\otimes b_{1})=a_{0}a_{0}'\otimes b_{0}b_{1}'-a_{1}a_{0}'\otimes b_{1}b_{1},
\]
If $char(K)\neq 2$, item $(ii)$ of the definition of regular algebra is not verified, then the twisted tensor product is not regular. 

However, it can be verified that $A \overline{\otimes} A$ is a regular algebra when one considers the decomposition given by
\[
A\overline{\otimes }A= (A_{0}\otimes A_{0})\oplus (A_{0}\otimes A_{1})\oplus (A_{1}\otimes A_{0})\oplus (A_{1}\otimes A_{1}).
\]     
Indeed, if we set $A_{(i,j)}:= A_{i}\otimes A_{j}$, then $A\overline{\otimes} A$ becomes a $\mathbb{Z}_{2}\times \mathbb{Z}_{2}$-graded regular algebra with bicharacter given by
\[
\beta((i,j),(k,l))=(-1)^{jk-il}.
\]
\end{Example}

Recall that two PI algebras $A$ and $B$ are PI equivalent if they satisfy the same polynomial identities, and we write $A \sim B$. The decomposition matrix of a regular algebra is a powerful tool for studying PI equivalences between certain regular algebras.

\begin{Theorem}[\cite{regevz2}, Theorem 3.1]
\label{teo.permut.matrix}
Let $A$, $B$ be algebras with regular decompositions $A=\oplus_{i=1}^{r} A_{i}$ and $B=\oplus_{i=1}^{r} B_{i}$ with corresponding decomposition matrices $M^{A}$ and $M^{B}$.  Assume that $M^{B}=PM^{A}P^{-1}$ where $P$ is a permutation matrix. Then $A$ and $B$ satisfy the same multilinear identities (hence they are PI equivalent when $char(K)=0$). 
\end{Theorem}
The $\xi$-gradings were studied in detail by Bahturin, Regev, and Zeilberger in \cite{bcommutation}. The following results are of significant relevance in our work. 

\begin{Theorem}
[\cite{bcommutation}, Theorem 2.2] 
\label{commutation.regev}
Let $\xi$ be an $n$-th root of unity (not necessarily primitive), and consider the matrix $\mathcal{M}=\Big(\xi^{jk-il}\Big)_{(i,j),(k,l)}$. If $V(x)=(x_{i}^{j-1})_{i,j}$ denotes the $n\times n$ Vandermonde matrix, and $\sigma$ is a permutation on the set $\{(i,j)\mid 1\leq i,j\leq n\}$ defined by $\sigma(i,j)=(j,i)$, then $\sigma$ induces naturally a row permutation on matrices of size $n^2$, denoted by $P_{\sigma}$, such that
    \[
    \mathcal{M}=D(\xi,\xi^{-1}),\quad \text{where}\quad D(\xi,\xi^{-1}):=P_{\sigma}(V(\xi)\otimes V(\xi^{-1})).
    \] 
    Here $V(\xi)\otimes V(\xi^{-1})$ stands for the Kronecker product of the two matrices.
\end{Theorem} 

We have the next.

\begin{Theorem} [\cite{bcommutation}, Proposition 2.3]
\label{teo.pauli.grading}
a) Let $\xi$ be an $n$-th root of unity (not necessarily primitive). Then
\begin{equation}
\label{eq.2.6}
\det D(\xi,\xi^{-1})=\pm \xi^{\Big(\dfrac{(n-1)^{2}(n-n^{2})}{2}\Big )}\Big(\prod_{i=1}^{n-1}(1-\xi^{i})\Big).
\end{equation}
b) If $K$ is algebraically closed, $char(K)=0$, and $\xi$ is a primitive $n$-th root of unity, then
\[
\det M^{M_{n}(K)}=\det D(\xi,\xi^{-1})=\pm n^{n^{2}}.
\]
\end{Theorem}
\subsection{Minimal decomposition of a regular algebra}
Now we recall the definition of minimal decomposition of a regular algebra,  introduced by Bahturin and Regev in \cite{bahturin2009graded}. 
\begin{Definition}
\label{def.min}
Given a $G$-graded regular algebra $A$ with bicharacter $\beta$, the regular decomposition of $A$ is called nonminimal if there exist $g$, $h\in G$ such that $\beta(x,g)=\beta(x,h)$ for every $x\in G$. 
\end{Definition}

In this case, we can replace, in the regular decomposition of $A$, the two summands $A_{g}$ and $A_{h}$ by the single summand $A_{g}\oplus A_{h}$, and the new decomposition is still regular, but the decomposition matrix has a smaller size than $M^{A}=(\beta(g,h))_{g,h}$. If $A$ is a $G$-graded regular algebra and there do not exist $g$, $h\in G$ as above, we say that the regular decomposition of $A$ is minimal.

Below we provide an example of a regular algebra whose regular decomposition is not minimal. 
\begin{Example} Consider $G=\mathbb{Z}_{2}\times \mathbb{Z}_{2}$ and let $\beta\colon G\times G\rightarrow K^{\ast}$ be the map defined by 
\[
\beta((i,j),(k,l))=(-1)^{ki+jl+jk-il},\quad \text{for every}\quad ((i,j),(k,l))\in G.
\]
We show below, in Proposition \ref{lem.class}, that $\beta$ is indeed a bicharacter of $G$. Moreover, it will be shown that there exists a $G$-graded regular algebra $A$ with bicharacter $\beta$ (see Theorem \ref{teo.permut.matrix}). Then we can write
\[
A=A_{(0,0)}\oplus A_{(0,1)}\oplus A_{(1,0)}\oplus A_{(1,1)},
\]
where the decomposition matrix associated to this regular decomposition is given by
\[
M^{A}=\begin{pmatrix}
    1 &  1 &  1 & 1\\
    1 &  -1 &  -1  & 1\\
    1 &  -1 &  -1 & 1\\
    1 &  1  &  1 & 1
\end{pmatrix}.
\]
This means the regular decomposition of $A$ is not minimal. On the other hand, if we set $B_{0}:=A_{(0,0)}\oplus A_{(1,1)}$ and $B_{1}:=A_{(0,1)}\oplus A_{(1,0)}$, then $A$ turns out to be the $\mathbb{Z}_{2}$-graded algebra $A=B_{0}\oplus B_{1}$ whose decomposition matrix is given by 
$\begin{pmatrix}
    1 & 1\\
    1 & -1
\end{pmatrix}$ and, if $char(K)=0$, it follows from Theorem \ref{teo.permut.matrix} that $A\sim E$.

\end{Example}
 Now we state the conjecture due to Bahturin and Regev's about the minimality of a regular decomposition of a $G$-graded regular algebra.
 
    \begin{Conjecture}
    [\cite{bahturin2009graded}, Conjecture 2.5] 
        \label{conjecture}
        Let $A$ be a $G$-graded regular algebra with regular decomposition $A=\oplus_{g\in G} A_{g}$. Then the regular decomposition of $A$ is minimal if and only if $\det M^{A}\neq 0$.
    \end{Conjecture}
    
If the regular decomposition of $A$ is not minimal, there exist $g$ and $h \in G$ such that $\beta(g,x) = \beta(h,x)$ for every $x \in G$. In other words, $M^{A}$ has two identical columns; in particular, $\det M^{A} = 0$. The more challenging part is to show that if $A$ is minimal, then $\det M^{A} \neq 0$. In \cite{Eli1}, Aljadeff and David provided a positive answer to this conjecture if $K$ is an algebraically closed field and $char(K)=0$. More precisely, they obtained the following.

\begin{Theorem}[\cite{Eli1}, Theorem 7] Let $A$ be a $G$-graded regular algebra, then, if the decomposition of $A$ is minimal, we have $\det M^{A}=\pm |G|^{|G|/2}$.
\end{Theorem}
It is worth mentioning that in \cite{Eli1}, the authors showed that the number of terms in a regular decomposition of a regular algebra $A$ is precisely $\exp(A)$, the PI exponent of the PI algebra $A$. As mentioned above, the PI exponent is one of the most important numerical invariants of a PI algebra.

In section 5 we will show that this conjecture may fail if $K$ has characteristic $p>2$.  

\section{\texorpdfstring{$(\mathbb{Z}_{n}\times \mathbb{Z}_{n})$-graded regular algebras}{[(Zn x Zn)-graded regular algebras}}

Let $A$ be a $(\mathbb{Z}_{n}\times \mathbb{Z}_{n})$-graded regular algebra, $n\geq 2$, with bicharacter $\beta$, where 
\begin{equation}
\label{eq.3.1}
A=\oplus_{(i,j)\in \mathbb{Z}_{n}\times \mathbb{Z}_{n}} A_{(i,j)}.
\end{equation} 
We will write the elements of $\mathbb{Z}_{n} \times \mathbb{Z}_{n}$ in the following order
\[
\mathbb{Z}_{n}\times \mathbb{Z}_{n}=\{(0,0),(0,1),(0,2),\ldots,(0,n-1), (1,1),\ldots, (1, n-1), \ldots, (n-1,n-1)\}.
\] 

\subsection{Technical results on bicharacters }

\begin{Lemma}
    \label{lem.3.1}
    For every $0 \leq i,j,k,l \leq n-1$, $\beta((i,j),(k,l))$ is an $n$-th root of unity. Moreover, $\beta((i,i),(i,i)) = \pm 1$ for every $0 \leq i \leq n-1$.
\end{Lemma}
\begin{proof}
   It follows from (\ref{eq.2.3}) that
    \[
    \beta((i,j),(k,l)) = \beta((i,j) + n(i,j),(k,l)) = \beta((i,j),(k,l))^{n+1},
    \]
    and since $\beta((i,j),(k,l)) \neq 0$, we have $\beta((i,j),(k,l))^{n} = 1.$ 
\end{proof}

The next result describes all possible bicharacters of $A$.

\begin{Proposition}
\label{lem.class}
Setting $\xi:=\beta((0,1),(1,0))$, $e:= \beta((0,1),(0,1))$ and $f:=\beta((1,0),(1,0))$, then $\beta$ has the form
\[
\beta((i,j),(k,l))=e^{ki}f^{jl}\xi^{jk-il}.
\]  
\end{Proposition}
\begin{proof}
    Note that since $(i,j) = i(1,0) + j(0,1)$, and $(k,l) = k(1,0) + l(0,1)$, it follows  
\begin{equation}
\label{eq.3.8}
\beta((i,j),(k,l)) = \beta(i(1,0) + j(0,1),(k,l)) = \beta((1,0),(k,l))^{i} \beta((0,1),(k,l))^{j}.
\end{equation}
On the other hand,
\begin{align*}
\beta((1,0),(k,l)) &= \beta((1,0),k(1,0) + l(0,1)) \\
&= \beta((1,0),k(1,0)) \beta((1,0),l(0,1)) = f^{k} \xi^{-l}
\end{align*}
and
\begin{align*}
\beta((0,1),(k,l)) &= \beta((0,1),k(1,0) + l(0,1)) \\
&= \beta((0,1),k(1,0)) \beta((0,1),l(0,1)) = e^{l} \xi^{k}.
\end{align*}
Thus, equation (\ref{eq.3.8}) becomes
\[
\beta((i,j),(k,l)) = e^{ki} f^{jl} \xi^{-il} \xi^{kj} = f^{ki} e^{jl} \xi^{jk-il}. \qedhere
\]
\end{proof}

\begin{Lemma}
    \label{lem.3.2}
    If $n$ is odd, then $e = f = 1$.
\end{Lemma}
\begin{proof}
    Indeed, as previously noted, $\beta((i,i),(i,i)) =\pm 1$, and thus, since 
    \[
    \beta((i,i),(i,i)) = \beta((i,i) + n(i,i),(i,i)) = \beta((i,i),(i,i))^{n+1},
    \]
    it follows from Proposition~\ref{lem.class} that $\beta((i,i),(i,i))^{n} = 1$. Therefore, in the case $n$ is odd,  we get that $\beta((i,i),(i,i)) = 1$.
\end{proof}

\subsection{The determinant of the decomposition matrices}

Keeping the notation above, we have the following result concerning the determinant of the decomposition matrix of the regular decomposition of $A$.
\begin{Theorem}
\label{det.decomposi.equal}
$\det M^{A}=\det D(\xi,\xi^{-1})$. 
\end{Theorem} 
\begin{proof}The proof will be carried out by exhausting all possible cases for $e$ and $f$.

If $e=f=1$, by Proposition~\ref{lem.class}, the bicharacter is given by $\beta((i,j),(k,l))=\xi^{jk-il}$. It follows that $\det M^{A}=\det D(\xi,\xi^{-1})$ (see also Example \ref{pauli.matrix}) and we are done. 

Now we can assume $e$ and $f$ are not simultaneously equal to $1$. Also, without loss of generality, by Proposition \ref{lem.3.2}, we can assume $n$ is even.

If $e=-1$ and $f=1$, the bicharacter associated to the regular decomposition of $A$ is given by
\[
	\beta((i,j),(k,l))=(-1)^{jl}\xi^{jk-il}.
\]
Let us recall that if $B = (b_{uv})_{uv} \in M_r(K)$, then
\begin{equation}
\label{eq.3.9}
\det B = \sum_{\sigma \in S_r} (-1)^{\sigma} b_{1\sigma(1)} b_{2\sigma(2)} \cdots b_{r\sigma(r)}.
\end{equation}
If we denote by $S$ the permutation group on the set $\{0, 1, \dots, n-1\}$, given $\sigma \times \sigma' \in S \times S$, we can define $A_{\sigma \times \sigma'}$ as the product 
\[
\beta((0,0), (\sigma(0), \sigma'(0))) \beta((0,1), (\sigma(0), \sigma'(1))) \cdots \beta((n-1,n-1), (\sigma(n-1), \sigma'(n-1)))
\]
hence, because of (\ref{eq.3.9}), we have
\begin{equation}
\label{eq.3.10}
\det M^{A}=\sum_{\sigma\times \sigma'\in S\times S} A_{\sigma\times \sigma'}. 
\end{equation}

If $\widetilde{\beta}$ is the bicharacter given by $\widetilde{\beta}((i,j),(k,l))=\xi^{jk-il}$, then there exist $p_{\sigma\times \sigma'}\in K[x]$, $\sigma\times \sigma'\in S\times S$, such that
\[
\det\Big[\Big(\widetilde{\beta}((i,j),(k,l))\Big)_{(i,k),(k,l)}\Big]=\sum_{\sigma \times \sigma' \in S \times S} p_{\sigma \times \sigma'}(\xi).
\]
Now, by Theorem \ref{commutation.regev}, $\det\Big[\Big(\widetilde{\beta}((i,j),(k,l))\Big)_{(i,k),(k,l)}\Big]=\det D(\xi,\xi^{-1})$, and we obtain
\[
\sum_{\sigma \times \sigma' \in S \times S} p_{\sigma \times \sigma'}(\xi)=\det D(\xi,\xi^{-1}).
\]
If $\gamma_{\sigma \times \sigma'} = \sum_{j=1}^{n-1} j \sigma'(j)$, we have
\[
A_{\sigma\times \sigma'}= \Big(\underbrace{(-1)^{\gamma_{\sigma\times \sigma'}}}_{(0,*)}\ast\Big)\Big(\underbrace{(-1)^{\gamma_{\sigma\times \sigma'}}}_{(1,*)}\ast\Big)\cdots \Big(\underbrace{(-1)^{\gamma_{\sigma\times \sigma'}}}_{(n-1,*)}\ast\Big),
\]
that is
\[
A_{\sigma \times \sigma'} = (-1)^{n\gamma_{\sigma \times \sigma'}} p_{\sigma \times \sigma'}(\xi).
\]
On the other hand, since $n$ is even, $n\gamma_{\sigma \times \sigma'}$ is even too, hence, $A_{\sigma \times \sigma'} = p_{\sigma \times \sigma'}(\xi)$. Consequently,
\[
\det M^{A} = \sum_{\sigma \times \sigma' \in S \times S} A_{\sigma \times \sigma'} = \sum_{\sigma \times \sigma' \in S \times S} p_{\sigma \times \sigma'}(\xi) = \det D(\xi,\xi^{-1}).
\]
The case when $e=1$ and $f=-1$ is analogous to the second one, and thus, similarly, we obtain $\det M^{A}=\det D(\xi,\xi^{-1})$.

Finally, in the last case, when $e=f=-1$, the bicharacter is given by
\[
\beta((i,j),(k,l)) = (-1)^{ki + jl}\xi^{jk - il}.
\]
Consider the same $S$, $A_{\sigma \times \sigma'}$, and $p_{\sigma\times \sigma'}$ as above.  In this way we get
\[
A_{\sigma \times \sigma'} = (-1)^{\theta_{\sigma \times \sigma'}}p_{\sigma \times \sigma'}(\xi)
\] 
%Now we will describe the form of $\theta_{\sigma \times \sigma'}$. With an argument similar to that of the case when $e=-1$ and $f=1$, the first elements of $A_{\sigma \times \sigma'}$, i.e., those associated to pairs of the form $(0,\ast)$, give rise to $\gamma_{\sigma \times \sigma'}$ (the same as defined in the second case). The subsequent terms, associated to pairs of the form $(1,\ast)$, give rise to $1\sigma(1)n+\gamma_{\sigma \times \sigma'}$. Proceeding in this way, the terms of the form $(r,\ast)$ give rise to $r\sigma(r)n +\gamma_{\sigma \times \sigma'}$. Consequently
and
\[
\theta_{\sigma \times \sigma'} = \gamma_{\sigma\times \sigma'} + (\sigma(1)n + \gamma_{\sigma \times \sigma'}) + \cdots + ((n-1)\sigma(n-1)n + \gamma_{\sigma \times \sigma'}).
\] 
Hence
\[
\theta_{\sigma \times \sigma'} = n\Big(\sum_{t=1}^{n-1}t\sigma(t)\Big) + n\gamma_{\sigma \times \sigma'} = n\Big(\sum_{t=1}^{n-1}t\sigma(t) + \gamma_{\sigma \times \sigma'}\Big).
\]
Since $n$ is even, $\theta_{\sigma \times \sigma'}$ is even as well. Thus $A_{\sigma \times \sigma'}= (-1)^{\theta_{\sigma\times \sigma'}}p_{\sigma\times \sigma'}(\xi)=p_{\sigma\times \sigma'}(\xi)$, then
\[
\det M^{A} = \sum_{\sigma \times \sigma' \in S \times S} A_{\sigma \times \sigma'}= \sum_{\sigma \times \sigma' \in S \times S}p_{\sigma \times \sigma'}(\xi) = \det D(\xi,\xi^{-1})
\] and the proof is complete.\qedhere
\end{proof}

\section{An explicit counterexample}

In this section, we present an explicit counterexample for Conjecture  \ref{conjecture}, when the field $K$ is algebraically closed and of characteristic $p>2$. First we construct the relatively free graded regular algebra.  

\subsection{The relatively free graded algebra}

Given a finite group $G$ and a bicharacter $\beta\colon G\times G\to K^{\ast}$, Aljadeff and David in \cite{Eli1} showed that it is possible to construct a $G$-graded regular algebra with $\beta$ as its associated bicharacter. For this purpose, they constructed the relatively free graded algebra. It happens that when the group $G$ is abelian, one can obtain a more explicit construction of this algebra. We will do this here; we also provide an example which shows that this algebra is not quite effective from the point of view of studying the ordinary polynomial identities of the algebra. 

Let $G$ be a finite abelian group and $\beta\colon G\times G\rightarrow K^{\ast}$ be a bicharacter of $G$. Consider the free $G$-graded algebra $K\langle X_{G}\rangle$. Recall that the integer degree (or simply the degree) of a monomial in $K\langle X_{G}\rangle$, denoted by $\deg_{\mathbb{Z}}(\cdot)$, is defined as follows: If $w$ is the monomial $w = x_{i_{1}}^{(g_{1})} \cdots x_{i_{n}}^{(g_{n})}$, then $\deg_{\mathbb{Z}}(w)=n$.  Also, the homogeneous degree of $w=x_{i_{1}}^{(g_{1})} \cdots x_{i_{n}}^{(g_{n})}$ is defined as $\deg (w)=g_{1}+\cdots+g_{n}$. 

Denote by $\Omega_{G}$ the following family of polynomials in $K\langle X_{G}\rangle$:
\begin{equation}
\label{free.graded}
\Omega_{G} = \{x_{i}^{(g)}x_{j}^{(h)}-\beta(g,h)x_{j}^{(h)}x_{i}^{(g)}\mid g,h\in G,\quad i,j\in \mathbb{N}\}.
\end{equation}
Let $I$ be the $T_{G}$-ideal of $K\langle X_{G}\rangle$ generated by $\Omega_{G}$. By definition, if $\mathscr{E}_{G}$ stands for the set of the graded endomorphisms of $K\langle X_{G}\rangle$, every element $y \in I$ is of the form
\[
y =\sum \alpha_{i,j}^{(g,h,f)} z_{i,j} \big( f_{i}^{(g)} f_{j}^{(h)} - \beta(g,h) f_{j}^{(h)} f_{i}^{(g)} \big)w_{ij},
\]
where $z_{i,j}, w_{i,j} \in K\langle X_{G}\rangle$, $f \in \mathscr{E}_{G}$, $\alpha_{i,j}^{(g,h,f)} \in K$, and $f_{i}^{(g)}$ and $f_{j}^{(h)}$ are the images of $x_{i}^{(g)}$ and $x_{j}^{(h)}$, respectively, under the graded endomorphism $f$.

In $K\langle X_{G}\rangle$, modulo the $T_{G}$-ideal $I$, we can commute two elements, up to scalars determined by $\beta(\cdot,\cdot)$. We denote with the same letter $x$ the image of an element $x\in X_{G}$ under the projection $K\langle X_{G}\rangle\to  K\langle X_{G}\rangle/I$. It follows that the homogeneous component of degree $g \in G$ in the $G$-graded algebra 
\[
B:= K\langle X_{G}\rangle/I=\bigoplus_{g\in G}B_{g}
\]
is spanned by monomials of the form
\[
 w=x_{i_{1}}^{(g_{i_{1}})}\cdots x_{i_{n}}^{(g_{i_{n}})}\quad i_{1}\leq \cdots\leq  i_{n},\quad g_{1}+\cdots+g_{n}=g\quad n\in \mathbb{N}.
\]

The proof of the next theorem can be found for finite (possibly non abelian) groups in \cite[Proposition 29]{Eli1}.

\begin{Theorem}
\label{uni.reg.}
The $G$-graded algebra $B$ defined above is regular, and $\beta$ is its associated bicharacter.
\end{Theorem}

\begin{Remark} Notice that, in the construction of $B$, one cannot assume that $I$ is the $T$-ideal (in the ordinary sense) generated by $\Omega_{G}$. Indeed,  suppose that $I$ is the $T$-ideal generated by $\Omega_{G}$ and assume that there exist $g$, $h \in G$ such that $\beta(g,h)=-1$.

Given any $i$, $j\in \mathbb{N}$, let $z =x_{i}^{(g)}$, $w =1$, and let $\phi$ be an endomorphism of $K\langle X\rangle$ satisfying
\[
\phi(x_{i}^{(g)})=x_{j}^{(h)}\quad \text{and}\quad  \phi(x_{j}^{(h)})=1.
 \]
 In this way, we have
 \[
\dfrac{1}{2}z\Big(\phi(x_{i}^{(g)})\phi(x_{j}^{(h)})-\beta(g,h)\phi(x_{j}^{(h)})\phi(x_{i}^{(g)})\Big)w=\dfrac{1}{2}(x_{i}^{(g)}x_{j}^{(h)}+x_{i}^{(g)}x_{j}^{(h)})=x_{i}^{(g)}x_{j}^{(h)}.
\]
However, since
\[
\dfrac{1}{2}z\Big(\phi(x_{i}^{(g)})\phi(x_{j}^{(h)})-\beta(g,h)\phi(x_{j}^{(h)})\phi(x_{i}^{(g)})\Big)w\in I
\]
we conclude that $x_{i}^{(g)}x_{j}^{(h)}\in I$. 
Therefore, in this case, $B$ fails to satisfy the first condition of regularity.

\end{Remark}

\subsection{Constructing the counterexample }

In this subsection we will assume that $K$ is an algebraically closed field of characteristic $p>2$. 

If $k$, $m\in \{0, \ldots,p-1\}$, with $k$ even and $m$ odd, then $(-1)^{k} \neq (-1)^{m}$. We shall use the standard notation $(m,n)$ to denote the greatest common divisor of the integers $n$ and $m$.

\begin{Theorem} There exist $n\in \mathbb{N}$ and a $(\mathbb{Z}_{n}\times \mathbb{Z}_{n})$-graded regular algebra $B$ such that  the regular decomposition of $B$ is minimal whereas $\det M^{B}=0$.
\end{Theorem}
\begin{proof}
    Let $t>2$ be a prime integer coprime with $p$. Now, consider the polynomial $h(x)=x^{t}-1$. Since $p$ does not divide $t$, $h'(x)=tx^{t-1}=0$ if and only if $x=0$, and therefore $h(x)$ is a separable polynomial over $K$. Since $K$ is algebraically closed, $h(x)$ has $t$ distinct roots in $K$, which are $t$-th roots of unity. The set of these roots forms a cyclic group $U_{t}\subseteq K^{\ast}$. Let $\zeta$ denote a generator of $U_{t}$, since $t$ is prime, $\zeta$ is a primitive $t$-th root of unity.

Now, consider $n:=2t$ and define
\[
\beta\colon (\mathbb{Z}_{n}\times \mathbb{Z}_{n})\times (\mathbb{Z}_{n}\times \mathbb{Z}_{n})\rightarrow K^{\ast},\quad ((i,j),(k,l))\mapsto (-1)^{ik+jl}\zeta^{jk-il}.
\]
The function $\beta$ is a bicharacter of $\mathbb{Z}_{n}\times \mathbb{Z}_{n}$. In fact,
\[
\beta((i,j)+(i',j'),(k,l))=(-1)^{(i+i')k+(j+j')l}\zeta^{(j+j')k - (i+i')l},
\]
and hence,
\begin{align*}
\beta((i,j)+(i',j'),(k,l))& =(-1)^{ik+jl} \zeta^{jk-il}(-1)^{i'k+j'l}\zeta^{j'k - i'l}\\
&=\beta((i,j),(k,l))\beta((i',j'), (k,l)).
\end{align*}
In a similar way, it can be shown that 
\[
\beta((i,j),(k,l)+(k',l'))=\beta((i,j),(k,l))\beta((i,j),(k',l')).
\]
 Finally,
\[
\beta((k,l),(i,j))=(-1)^{ki+jl}\zeta^{il-jk}=\beta((i,j), (k,l))^{-1}.
\]
Hence, it follows from  Theorem \ref{uni.reg.} that there exists a $(\mathbb{Z}_{n}\times \mathbb{Z}_{n})$-graded regular algebra $B$ with $\beta$ as its associated bicharacter. 

But $r$, $s\in \{0,1,\ldots, n-1\}$ satisfy
$r\equiv s\pmod{t}$ if and only if $r-s=t$. Indeed, otherwise, if $r=s+kt$, with $k>1$, then $r-s=kt\geq 2t=n$, a contradiction. In particular, since $t$ is odd, $r$ and $s$ must have different parities.

Suppose that there are two equal columns in $M^{B}$, say $(r,s)$ and $(r',s')$, then
\[
(-1)^{ri+js}\zeta^{jr-is}=(-1)^{r'i+js'}\zeta^{jr'-is'},\quad \text{for every }\quad (i,j).
\]
In particular, analyzing respectively the entries related to rows $(0,2)$ and $(2,0)$, we have
\[
(-1)^{2s}\zeta^{2r}=(-1)^{2s'}\zeta^{2r'}, \qquad 
(-1)^{2r}\zeta^{-2s}=(-1)^{2r'}\zeta^{-2s'}
\]
which implies that $r\equiv r'\pmod{t}$ and $s\equiv s'\pmod{t}$.

As a consequence, $r$ and $r'$ have different parities (and so do $s$ and $s'$). On the other hand, the equality $\beta((1,0),(r,s))=\beta((1,0),(r',s'))$ is equivalent to
\[
(-1)^{r}\zeta^{-s}=(-1)^{r'}\zeta^{-s'}
\]
and since $\zeta^{-s}=\zeta^{-s'}$, it follows that $(-1)^{r}=(-1)^{r'}$.  But this is a contradiction since $r$ and $r'$ have different parities. Thus, $M^{B}$ cannot have two equal columns, and consequently, the regular decomposition of $B$ is minimal. Nonetheless we will show that $\det M^{B}=0$. As seen previously, by Theorem \ref{det.decomposi.equal} we have $\det M^{B}=\det D(\zeta,\zeta^{-1})$, and so it follows from Theorem \ref{teo.pauli.grading} that

\[
\det M^{B}=\zeta^{\Big(\dfrac{(n-1)^{2}(n-n^{2})}{2}\Big )}\Big(\prod_{i=1}^{n-1}(1-\zeta^{i})\Big).
\]
Since $\zeta^{t}=1$ we conclude that $\det M^{B}=0$
\end{proof}

\section{The counterexample for the conjecture in the general case}

We assume, as above, that $K$ is a field of characteristic $p>2$. In this section we show that there exists a finite abelian group $G$ and a suitable choice of a cocycle $\alpha \in H^{2}(G, K^{\ast})$ such that the twisted group algebra $K^{\alpha}G$ has a minimal regular decomposition, yet $\det M^{K^{\alpha}G}=0$. First we recall several results concerning twisted group algebras.

\subsection{The twisted group algebra}

The twisted group algebra was defined in Example \ref{twisted.group}. A fundamental result about such algebras is Scheunert's Lemma, which enables us to construct cocycles from specific bicharacters. 

\begin{Theorem}[\cite{scheunert}, Lemma 2] 
\label{scheunert}
Let $G$ be a finite abelian group. Given a bicharacter $\beta\colon G\times G\rightarrow K^{\ast}$ with $\beta(g,g)=1$ for every $g\in G$, there exists $\alpha\in H^{2}(G,K^{\ast})$ such that 
 \[
\beta(g,h)=\alpha(g,h)\alpha(h,g)^{-1} \quad \text{for every}\quad g,h\in G.
 \]
 \end{Theorem}
Let $G$ be a finite abelian group, $\beta\colon G \times G \to K^{\ast}$ a bicharacter of $G$ satisfying $\beta(g,g)=1$, for every $g\in G$. It follows from Theorem \ref{scheunert}, that there exists $\alpha \in H^{2}(G, K^{\ast})$ such that the regular $G$-graded algebra $B:= K^{\alpha}G$ has $\beta$ as the bicharacter associated with its regular decomposition. It is equivalent to say that $\beta$ is the bicharacter induced by $\alpha$. By considering representatives in $\alpha$, without loss of generality, we  assume that $\alpha(g,0)=\alpha(0, g)=\alpha(0,0)=1$, for every $g\in G$. 

\begin{Remark}
\label{remark.inverse}
By definition, $x_{0}$ is the unit of $K^{\alpha}G$. Note that if $g\in G$, then 
\[
x_{g}^{-1}=\alpha(g,g^{-1})^{-1}x_{g^{-1}}
\]
 In fact, by the multiplication rule we get
\[
x_{g}\Big(\alpha(g,g^{-1})^{-1}x_{g^{-1}}\Big)=\alpha(g,g^{-1})^{-1}\alpha(g,g^{-1})x_{0}=x_{0}.
\]
On the other hand, we have
\[
\alpha(g,g^{-1})=\alpha(g^{-1},0)\alpha(g,g^{-1})=\alpha(g^{-1},g+g^{-1})\alpha(g,g^{-1}),
\]
and by the cocycle condition (see Example \ref{twisted.group}), we obtain
\[
\alpha(g,g^{-1})=\alpha(g^{-1},g+g^{-1})\alpha(g,g^{-1})=\alpha(g^{-1}+g, g^{-1})\alpha(g^{-1},g)=\alpha(g^{-1},g).
\]
Analogously one proves that $\Big(\alpha(g,g^{-1})^{-1}x_{g^{-1}}\Big)x_{g}=x_{0}$.
\end{Remark}

Denote by $\mathcal{C}:=\{x_{g}\mid g\in G\}$ a basis of $B$. We need to describe an element in the centre $Z(B)$ of $B$. Given $z\in Z(B)$,  by definition, we have
\begin{equation}
\label{eq.5.2.1}
x_{g}z=zx_{g}, \quad \text{for every} \quad g\in G.    
\end{equation}
Writing $z=\sum_{w\in G}\mu_{w}x_{w}$, $\mu_{w}\in K$, the multiplication in $B$ implies that condition (\ref{eq.5.2.1}) is satisfied if and only if
\[
\alpha(w,g)=\alpha(g,w),\quad \text{for every}\quad g\in G.
\]
Therefore, setting
\[
Y=\{x_{w}\mid \alpha(g,w)=\alpha(w,g),\quad \text{for every}\quad g\in G \},
\] 
it follows that $Z(B)=\text{\rm span}_{K}(Y)$.

Notice that if $B$ is simple and noncommutative (there exist $s\in G$ and $x\in G$ such that $\alpha(s,x)\neq \alpha(x,s)$) then $\dim Z(B)=1$. Indeed, if $x_{w}\in Z(G)$ and $w\neq 1$, since $B$ is simple, if $I$ is the ideal generated by $x_{w}$, it follows that $I=B$. Now consider $x_{g}x_{w}\in I$ and take any element $h\in G$. Observe that
 \[
(x_{g}x_{w})x_{h}=x_{g}(x_{w}x_{h})=\alpha(w,h)\alpha(g,w+h)x_{g+w+h}.
\]
On the other hand,
\[
x_{h}(x_{g}x_{w})=\alpha(g,w)\alpha(h,g+w)x_{g+w+h}.
\]
By the cocycle condition it follows that $(x_{g}x_{w})x_{h}=x_{h}(x_{g}x_{w})$, then $I$ is commutative. 
 Consequently, $B$ is commutative, but this is a contradiction and we get $Z(B)=\text{\rm span}_{K}\{x_{0}\}$. 

\begin{Remark}
\label{remark.equiv.min}
According to Definition \ref{def.min}, if $B$ is nonminimal, then there exist $g$, $h\in G$ such that $\beta(g,x)=\beta(x,h)$, for every $x\in G$. As $\beta$ is induced by $\alpha$, this is equivalent to 
\[
\beta(g,x)=\beta(x,h)=1,\qquad \alpha(g,x)=\alpha(x,h),\quad \text{for every}\quad x\in G.
\]
As a consequence, $B$ has minimal regular decomposition if and only if $\dim Z(B)=1$.
\end{Remark}

\begin{Lemma}
\label{aux.const.examp.}
    If $B$ is simple and noncommutative, then the regular decomposition of $B$ is minimal.
\end{Lemma}
\begin{proof}
Since $B$ is simple and noncommutative, as we have seen above, it follows that $\dim Z(B) = 1$. Thus, by Remark \ref{remark.equiv.min}, the regular decomposition of $B$ is minimal.    
\end{proof}
\subsection{The counterexample for infinite fields}

We start with computing the determinant of the decomposition matrix of a twisted group algebra. To this end we apply a technique introduced by Aljadeff and David in \cite{Eli1}.

\begin{Lemma}
\label{average.sum}
Let $G$ be a finite abelian group and let $\beta\colon G \times G \rightarrow K^{\ast}$ be a nontrivial bicharacter of $G$ satisfying $\beta(g,g) = 1$ for every $g \in G$. If $0 \neq a \in G$, then $\sum_{g\in G}  \beta(g,a)=0$.
\end{Lemma}
\begin{proof}
    Define $\chi\colon G \rightarrow K^{\ast}$ by $\chi(g) = \beta(g, a)$. Since $\beta$ is a bicharacter, $\chi$ is a group homomorphism. Let $k \in G$ be such that $\beta(k, a) \neq 1$. Then
    \[
	\sum_{g}\chi(g)=\sum_{g}\chi(kg)=\chi(k)\sum_{g}\chi(g),
    \]
therefore, since $\chi(k) \neq 1$, it follows that $\sum_{g} \chi(g)=0$.
\end{proof}
\begin{Lemma}
\label{eq.equiv}
Let $G$ be a finite abelian group and $\alpha \in H^{2}(G, K^{\ast})$ be a cocycle. Denote by $B := K^{\alpha}G$ the twisted group algebra of $G$ with respect to $\alpha$. Then, $\det M^{B}=0$ if and only if $p$ divides $|G|$.
\end{Lemma}
\begin{proof}
    Let $\beta$ be the bicharacter induced by $\alpha$. Given $g$, $h \in G$, it follows 
    \[
x_{g}x_{h}x_{g}^{-1}=\beta(g,h)x_{h}x_{g}x_{g}^{-1}=\beta(g,h)x_{h}x_{0}=\beta(g,h)x_{h}.
    \]
By Remark \ref{remark.inverse} we have $x_{g}^{-1}=\alpha(g,g^{-1})^{-1}x_{g^{-1}}$, and as a consequence
\begin{align*}
x_{h}x_{g}^{-1}x_{h}^{-1} &=x_{h}(\alpha(g,g^{-1})^{-1}x_{g^{-1}})x_{h}^{-1}=\beta(h,g^{-1})\alpha(g,g^{-1})^{-1}x_{g^{-1}}\\
&=\beta(h,g^{-1})x_{g}^{-1}.
\end{align*}
Denote by $m_{(g,h)}$ the $(g,h)$-entry of the matrix $(M^{B})^{2}$. Thus
\begin{align*}
&m_{(g,h)}x_{0}=(\sum_{k\in G}\beta(g,k)\beta(k,h))x_{0}=(\sum_{k\in G} (x_{g}x_{k}x_{g}^{-1}x_{k}^{-1})(x_{k}x_{h}x_{k}^{-1}x_{h}^{-1}))x_{0}\\
&=(\sum_{k\in G}x_{g}(x_{k}(x_{g}^{-1}x_{h})x_{k}^{-1})x_{h}^{-1})x_{0}=(\sum_{k\in G}x_{g}(x_{k}x_{g}^{-1}x_{k}^{-1})(x_{k}x_{h}x_{k}^{-1})x_{h}^{-1}) x_{0}.
\end{align*}
It follows that
$m_{(g,h)}x_{0}=(x_{g}(\sum_{k\in G}\beta(k,g^{-1})\beta(k,h))x_{g}^{-1}x_{h}x_{h}^{-1})x_{0}$.
Hence
\begin{equation}
\label{eq.case p}
m_{(g,h)}x_{0}=\Big(\sum_{k\in G}\beta(k,g^{-1}h)\Big)x_{0}.
\end{equation}
Lemma \ref{average.sum}, applied to Eq.~(\ref{eq.case p}), gives us that
\[
m_{(g,h)}=\Big(\sum_{k}\beta(k,g^{-1}h)\Big)=\begin{cases}
0\quad  \text{if}\quad g\neq h\\
|G|\quad \text{if}\quad g=h
\end{cases}.
\]
We  showed that $(M^{B})^{2}=|G|I_{|G|}$, where $I_{|G|}$ is the  $|G|\times |G|$ identity matrix. It follows that $\det M^{B}= 0$ if and only if $p$  divides $|G|$. 
\end{proof}

Now, let $G$ be a finite abelian group and $\alpha\in H^{2}(G,K^{\ast})$ be a cocycle. Consider the following subset of $G$ (in \cite{Duarte.Polcino} it was called the set  of regular elements of $G$):
\[
G_{0}:=\{s\in G\mid \alpha(s,x)=\alpha(x,s),\quad \text{for every}\quad x\in G\}.
\]
It follows by Remark \ref{remark.equiv.min} that if $\beta$ is the bicharacter induced by the cocycle $\alpha$, then
\[
G_{0}=\{s\in G\mid \beta(s,x)=1,\quad \text{for all}\quad x\in G\}.
\]
It is easy to see that $G_{0}$ is a subgroup of $G$. Indeed, given $s$, $s'\in G_{0}$, and $x\in G$, we have
\[
\beta(s-s',x)=\beta(s,x)\beta(-s',x)=\beta(s',x)^{-1}=1.
\]
This implies that $s-s'\in G_{0}$. In particular, since $G$ is abelian, $G_{0}$ is a normal subgroup of $G$. 

We summarize the above facts in the following proposition.

\begin{Proposition}
\label{prop.equiv.min.}
 Let $G$ be a finite abelian group and $\alpha\in H^{2}(G,K^{\ast})$ be a cocycle. Denote by $\beta$ the bicharacter induced by $\alpha$. If $B:=K^{\alpha}G$, the following are equivalent:
\begin{enumerate}
    \item The regular decomposition of $B$ is minimal.
    \item  $\dim Z(B)=1$.
    \item $G_{0}=\{0\}$.
\end{enumerate}
\end{Proposition}

The next result shows us how to obtain a minimal decomposition of a twisted group algebra from a factor group of $G$. 

\begin{Proposition}
\label{construct.minimal}
    Let $\beta \colon G \times G \to K^{\ast}$ be a nontrivial bicharacter of $G$ satisfying $\beta(g,g)=1$, for every $g \in G$. 
 Then there exists a finite abelian group $\widetilde{H}$, which is  a quotient of $G$, and a regular $\widetilde{H}$-graded algebra $C$, whose bicharacter is uniquely determined by $\beta$, and the regular decomposition of $C$ is minimal.
\end{Proposition}
\begin{proof}
    By Scheunert's theorem, there exists a cocycle $\alpha \in H^{2}(G, K^{\ast})$ such that the regular algebra $A:= K^{\alpha}G$ has $\beta$ as its bicharacter. Put $\widetilde{H}:=G/G_{0}$, and define
    \[
    \overline{\beta}\colon \widetilde{H}\times \widetilde{H}\rightarrow K^{\ast},\quad (a+G_{0},b+G_{0})\mapsto \beta(a,b).
    \]
    We have to show that $\overline{\beta}\colon \widetilde{H}\times \widetilde{H}\rightarrow K^{\ast}$ is a bicharacter of $\widetilde{H}$. 

    First, suppose that $(a+G_{0},b+G_{0})=(a'+G_{0},b'+G_{0})$, in other words, there exist $r$, $s\in G_{0}$, such that
$a=a'+r$ and $b=b'+s$. This implies that 
    \[
    \overline{\beta}((a+G_{0},b+G_{0}))=\beta(a,b)=\beta(a'+r,b'+s)=\beta(a',b')\beta(a',s) \beta(r,b')\beta(r,s).
    \] 
    As $r$, $s\in G_{0}$, it follows $\beta(a,r) =\beta(a',s)= \beta(r,b')= \beta(r,s)=1$, and  we get
    \[
    \overline{\beta}((a+G_{0},b+G_{0}))=\beta(a',b')=\overline{\beta}((a'+G_{0},b'+G_{0})).
    \]
    This shows that $\overline{\beta}$ is well defined. On the other hand,  by the definition of $\overline{\beta}$ and the fact that $\beta$ is a bicharacter, it follows that for $a$, $a'$, $b$, $b'\in G$,
    \begin{align*}
    \overline{\beta}(a+a'+G_0,b+G_0)&=\overline{\beta}(a+G_{0},b+G_{0})\overline{\beta}(a'+G_{0},b+G_{0}),\\
    \overline{\beta}(a+G_0,b+b'+G_0)& =\overline{\beta}(a+G_{0},b+G_{0})\overline{\beta}(a+G_{0},b'+G_{0}).
    \end{align*}
Also, given $a$, $b\in G$ one has
    \[
    \overline{\beta}(a+G_{0},b+G_{0})=\beta(a,b)=\beta(b,a)^{-1}=\overline{\beta}(a+G_{0},b+G_{0})^{-1}.
    \]
Therefore, $\overline{\beta}$ is a bicharacter of $\widetilde{H}$. By Theorem \ref{scheunert}, there exists a cocycle $\widetilde{\alpha}\in H^{2}(\widetilde{H}, K^{\ast})$ such that the regular algebra $C:=K^{\widetilde{\alpha}}\widetilde{H}$ has $\overline{\beta}$ as its associated bicharacter. Proposition \ref{prop.equiv.min.} implies that the regular decomposition of $C$ is minimal. Indeed, if $s+G_{0}\in H$ satisfies
    $\overline{\beta}(s+G_{0},g+G_{0})=1$ for every $g\in G$, then
$\beta(s,g)=1$, for every $g\in G$.
This implies that $s\in G_{0}$, therefore $\widetilde{H}_{0}=\{0_{\widetilde{H}}\}$.    
\end{proof}

\begin{Lemma}
\label{construct.cocycle}
Let $G$, be a  finite non-cyclic abelian group. Then, there exists $\alpha\in H^{2}(G,K^{\ast})$ such that $G_{0}\neq G$. Moreover, if the number of cyclic factors in the decomposition of $G$ is greater than 2, it is possible to find such an $\alpha\in H^{2}(G,K^{\ast})$ such that $G_{0}$ is a proper subgroup of $G$.
\end{Lemma}
\begin{proof}
    Without loss of generality, we assume $G= \langle a_{1}\rangle\times  \langle a_{2}\rangle\times \cdots \times  \langle a_{r}\rangle$.
    Since $K$ is infinite,  one can choose distinct $\lambda_{1}$, $\lambda_{2}$, \dots, $\lambda_{r}\in K^{\ast}\setminus\{1\}$.

Denote by $n_{i}$ the order of $\langle a_{i}\rangle$, $1\leq i\leq r$. Now define 
    \[
    t_{\lambda_{1},\ldots,\lambda_{r}}(\sum_{j}i_{j}a_{j},\sum_{j}k_{j}a_{j})=\prod_{s=1}^{r}t_{\lambda_{s}}(i_{s}a_{s},k_{s}a_{s}),
    \]
where for  each $1\leq s\leq r$
    \[
    t_{\lambda_{s}}(i_{s}a_{s},k_{s}a_{s})=\begin{cases}   
    1\quad\text{if}\quad i_{s}+k_{s}<n_{s}\\
    \lambda_{s}\quad\text{if}\quad i_{s}+k_{s}\geq n_{s}
   \end{cases}.
    \]  
    As shown in \cite[p.~5]{Duarte.Polcino}, $t_{\lambda_{1},\ldots,\lambda_{r}}$ is an element of $H^{2}(G,K^{\ast})$. Moreover, by definition
    \begin{equation}
    \label{eq. sym.}
    t_{\lambda_{1},\ldots,\lambda_{r}}(\sum_{j}i_{j}a_{j},\sum_{j}k_{j}a_{j})=t_{\lambda_{1},\ldots,\lambda_{r}}(\sum_{j}k_{j}a_{j},\sum_{j}i_{j}a_{j}).
    \end{equation}
    Next, consider 
    \[
    \alpha(\sum_{j}i_{j}a_{j},\sum_{j}k_{j}a_{j} )= (-1)^{i_{2}k_{1}}(t_{\lambda_{1},\ldots,\lambda_{r}}(\sum_{j}i_{j}a_{j},\sum_{j}k_{j}a_{j})).
    \]
We claim $\alpha\in H^{2}(G,K^{\ast})$. Indeed, since $t_{\lambda_{1},\lambda_{2},\ldots,\lambda_{r}}\in H^{2}(G,K^{\ast})$ it is  sufficient to show that
    \[
    \gamma\colon (\sum_{j}i_{j}a_{j},\sum_{j}k_{j}a_{j})\mapsto (-1)^{i_{2}k_{1}}\in H^{2}(G,K^{\ast}).
    \]
    In order to do this, we observe that
    \[
    \gamma(\sum_{j}i_{j}a_{j}+\sum_{j}q_{j}a_{j},\sum_{j}k_{j}a_{j} )\gamma(\sum_{j}i_{j}a_{j},\sum_{j}q_{j}a_{j})=(-1)^{(i_{2}+q_{2})k_{1}}(-1)^{i_{2}q_{1}},
    \]
    \[
\gamma(\sum_{j}i_{j}a_{j},\sum_{j}q_{j}a_{j}+\sum_{j}k_{j}a_{j} )\gamma(\sum_{j}q_{j}a_{j},\sum_{j}k_{j}a_{j})=(-1)^{i_{2}(q_{1}+k_{1})}(-1)^{q_{2}k_{1}}.
    \]
On the other hand, since
\begin{align*}
(-1)^{(i_{2}+q_{2})k_{1}}(-1)^{i_{2}q_{1}}& =(-1)^{i_{2}k_{1}+q_{2}k_{1}+i_{2}q_{1}}=(-1)^{i_{2}q_{1}+i_{2}k_{1}+q_{2}k_{1}}\\
&=(-1)^{i_{2}(q_{1}+k_{1})}(-1)^{q_{2}k_{1}},
\end{align*}
we have that $\gamma$ satisfies the cocycle condition and hence $\gamma\in H^{2}(G,K^{\ast})$. 
Consequently, $\alpha\in H^{2}(G,K^{\ast})$.  

Finally, in order to see that $G_{0}\neq G$, it is sufficient to exhibit two elements $x$, $y\in G$ with $\alpha(x,y)\neq\alpha(y,x)$. Consider $x=a_{1}+a_{3}$ and $y=a_{1}+a_{2}+a_{3}$. Observe that
\begin{align*}
\alpha(x,y)& =(-1)^{0\cdot 1}t_{\lambda_{1}\ldots, \lambda_{r}}(x,y)=t_{\lambda_{1}\ldots, \lambda_{r}}(x,y), \\
\alpha(y,x)&=(-1)^{1}t_{\lambda_{1}\ldots, \lambda_{r}}(y,x)=-t_{\lambda_{1}\ldots, \lambda_{r}}(x,y).
\end{align*}
By definition, we have $t_{\lambda_{1},\ldots, \lambda_{r}}(x,y)=t_{\lambda_{1},\ldots, \lambda_{r}}(y,x)$, then, because the characteristic of the base field is $p>2$, we have $\alpha(x,y)\neq \alpha(y,x)$.
This shows that $G_{0}\neq G$. 

Now, suppose $r>2$. If $x'=a_{3}$, it follows from (\ref{eq. sym.}) that for every $z\in G$
\[
\alpha(x',z)=(-1)^{0}t_{\lambda_{1},\ldots, \lambda_{r}}(x',z)= (-1)^{0}t_{\lambda_{1},\ldots, \lambda_{r}}(z,x')=\alpha(z,x')
\]
and so $x'\in G_{0}$, which shows that $G_{0}\neq \{0\}$.
\end{proof}

\begin{Remark} According to \cite[Corollary 2.4 ]{Duarte.Polcino}, the result of the last lemma does not hold for cyclic groups, i.e if $G$ is cyclic and $\alpha\in H^{2}(G,K^{\ast})$, then $G_{0}=G$.
\end{Remark}

\begin{Theorem}
\label{main.theorem.1}
Let $K$ be an infinite field of characteristc $p>2$. Then, there exists a finite abelian group $G$ and a regular $G$-graded algebra $A$ such that the regular decomposition of $A$ is minimal, but $\det M^{A}=0$
\end{Theorem}
\begin{proof}
    Let $Q=\mathbb{Z}_{p}\times \mathbb{Z}_{p}\times \mathbb{Z}_{p}$, it is an elementary abelian group of order $p^{3}$. By Lemma \ref{construct.cocycle}, there exists $\alpha\in H^{2}(Q,K^{\ast})$ such that $Q_{0}\notin \{\{0\},Q\}$.
    
   If $\beta$ is the bicharacter induced by the cocycle $\alpha$, it satisfies
$\beta(g,g)=1$, for every $g\in G$.

By the construction given in Proposition \ref{construct.minimal}, if $G:=Q/Q_{0}$, there exists $\delta\in H^{2}(G,K^{\ast})$ such that  $A=K^{\delta}G$ is a $G$-graded regular algebra with minimal regular decomposition. 
    
    However, since $Q_{0}\ne 0$, and $Q_0\ne Q$, we conclude that $p$ divides $|G|$, and it follows from Lemma \ref{eq.equiv} that $\det M^{A}=0$.
\end{proof}

\subsection{The counterexample for finite fields}

Here we  deal with the case when $K$ is a finite field of characteristic $p>2$. We recall that Lemma \ref{construct.cocycle} was crucial in the proof of Theorem \ref{main.theorem.1}. 

\begin{Lemma}
\label{lem.construc.finite}
Let $K$ be a finite field, then there exist an abelian group $(T,+)$ and a cocycle $\alpha\in H^{2}(T,K^{\ast})$ such that $T_{0}\notin \{\{0\},T\}$.  
\end{Lemma}
\begin{proof}
    We divide the proof into two cases. If $p=3$, consider 
    \[
    T =\langle a_{1}\rangle\times \langle a_{2}\rangle\times \langle a_{3}\rangle\cong \mathbb{Z}_{3}\times \mathbb{Z}_{3}\times \mathbb{Z}_{3}, 
    \]
 and set
    \[
    \alpha(\sum_{j=1}^3 i_{j}a_{j},\sum_{j=1}^3 k_{j}a_{j})=(-1)^{i_{2}k_{1}}t_{-1}(i_{3}a_{3},k_{3}a_{3}).
    \]
    By the same computation as in Lemma \ref{construct.cocycle}, it follows that $\alpha\in H^{2}(T,K^{\ast})$. Now observe that if $x=a_2+a_{3}$ and $y=a_{1}+a_{3}$, then 
    \[
    \alpha(x,y)=-t_{-1}(a_3,a_3), \quad \alpha(y,x)=t_{-1}(a_{3},a_{3}).
    \]
Since $p=3$, it follows that $\alpha(x,y)\neq \alpha(y,x)$. Therefore, $T_{0}\neq T$. In order to see that $T_{0} \neq \{0\}$, define $x' = a_{3}$, and observe that for every $z=k_{1}a_{1}+k_{2}a_{2}+k_{3}a_{3} \in T$, we have
\[
\alpha(x',z)=(-1)^{0}t_{-1}(a_{3},k_{3}a_{3})=(-1)^{0}t_{-1}(k_{3}a_{3},a_{3})=\alpha(z,x').
\]
Therefore $x'\in T_{0}$. 

If $p>3$, then the proof of Lemma \ref{construct.cocycle} works for the following abelian group $T$: 
\[
T=\langle a_{1}\rangle \times \langle a_{2}\rangle \times \langle a_{3}\rangle\cong \mathbb{Z}_{p} \times \mathbb{Z}_{p} \times \mathbb{Z}_{p}.
\]
 This is seen by considering, in Lemma \ref{construct.cocycle}, $\lambda_{1}=2$, $\lambda_{2}=3$, and $\lambda_{3}=4$.
\end{proof}

In a similar way to Theorem \ref{main.theorem.1}, we have the following
\begin{Theorem}
\label{main.theorem.2}
Let $K$ be a finite field of characteristic $p>2$. Then there exists a finite abelian group $G$ and a $G$-graded regular algebra $A$ such that the regular decomposition of $A$ is minimal, but $\det M^{A}=0$.
\end{Theorem}
\begin{proof} Consider $T=\mathbb{Z}_{p}\times \mathbb{Z}_{p}\times \mathbb{Z}_{p}$. By Lemma \ref{lem.construc.finite}, there exists $\alpha\in H^{2}(T,K^{\ast})$ such that $T_{0}\notin \{\{0\},T\}$. Now the proof follows in the same way as in Theorem \ref{main.theorem.1}.
\end{proof}

\subsection{A construction of a regular algebra that verifies the conjecture}

Here we construct a regular algebra that  verifies the Bahturin and Regev's conjecture for an arbitrary infinite field $K$ of characteristic $p>2$. In order to do this, it is sufficient to consider a finite abelian group that is the product of two or more cyclic groups and proceed in accordance with Proposition \ref{construct.minimal} and Lemma \ref{eq.equiv}. In the following example we use appropriate choices of roots of unity.

 Let $q$ and $q_1$ be distinct prime numbers, each different from $p$, such that $(p,q-1)=(p,q_{1}-1)=1$. Let $\zeta$ and $\zeta_1$ be primitive $q$-th and $q_1$-th roots of unity, respectively, in $\overline{K}$.

It is well-known that $K(\zeta)/K$ and $K(\zeta_1)/K$ are Galois abelian field extensions, i.e., $\text{Gal}(K(\zeta)/K)$ and $\text{Gal}(K(\zeta_1)/K)$ are abelian groups. In particular,
    \begin{equation}
    \label{eq.5.4.1}
    |\text{Gal}(K(\zeta)/K)|=[K(\zeta): K]= \phi(q)=q-1,
    \end{equation}
    \begin{equation}
    \label{eq.5.4.2}
    |\text{Gal}(K(\zeta_{1})/K)|=[K(\zeta_{1}): K]= \phi(q_{1})=q_{1}-1,
    \end{equation}
    where $\phi$  is the Euler's totient function. It follows from \cite[Corollary 1.15, Chapter VI]{lang}, that the Galois group of the composite field $E:=K(\zeta)K(\zeta_{1})$ is isomorphic to the following
    \[
    \text{Gal}(E/K)\cong \text{Gal}(K(\zeta)/K)\times \text{Gal}(K(\zeta_{1})/K).
    \]
Therefore by  Lemma \ref{construct.cocycle}, there exists $\alpha \in H^{2}(\text{Gal}(E/K),K^{\ast})$ such that 
    \[
    (\text{Gal}(E/K))_{0}\notin \{\{0\}\text{Gal}(E/K)\}.
    \]
    In particular, this implies that there exists $\alpha\in H^{2}(\text{Gal}(E/K),K^{\ast})$ such that the twisted group algebra $B:=K^{\alpha}\text{Gal}(E/K)$ is noncommutative. If $\widetilde{G}:=\text{Gal}(E/K)$, it follows from Proposition \ref{construct.minimal} that there exists $\widetilde{\alpha} \in H^{2}(\widetilde{G}/\widetilde{G}_{0}, K^{\ast})$, such that the twisted group algebra $C:=K^{\widetilde{\alpha}}(\widetilde{G}/\widetilde{G}_{0})$ is a $\widetilde{G}/\widetilde{G}_{0}$-graded regular algebra with minimal regular decomposition.

On the other hand, since $|\widetilde{G}|=(q-1)(q_{1}-1)$, it follows that 
\[
p\nmid \Big|\widetilde{G}/\widetilde{G}_{0}\Big|
\]
and, by Lemma \ref{eq.equiv}, we have $\det M^{C} \neq 0$. Thus, the $\widetilde{G}/\widetilde{G}_{0}$-graded regular algebra $C$ verifies the conjecture. 

This last example also shows us an interesting fact about the theory of twisted group algebras and crossed product algebras. Recall that if $\Omega/K$ is a Galois extension of $K$, and $\alpha \in H^{2}(\text{Gal}(\Omega/K), \Omega^{\ast})$, then we define the crossed product of $\Omega$ and $\text{Gal}(\Omega/K)$ as the $K$-algebra $(\Omega, \text{Gal}(\Omega/K), \alpha)$ with basis $\{x_{\sigma} \mid \sigma \in \text{Gal}(\Omega/K)\}$, whose product is defined on the basic elements by
\[
x_{\sigma}x_{\tau}:=x_{\sigma\tau}\alpha(\sigma,\tau),\quad \sigma,\tau \in \text{Gal}(\Omega/K),
\]
and the action of $\Omega$ on $(\Omega,\text{Gal}(\Omega/K),\alpha)$ satisfies
\[
t x_{\sigma}:=x_{\sigma}\sigma(t),\quad t\in \Omega.
\]
Let us return to our example. It is well known that $(E,\widetilde{G},\alpha)$ is a central simple algebra over $K$ \cite[Theorem 4.4.1]{herstein}, but this is not the case with $K^{\alpha}\widetilde{G}$. Otherwise, by Lemma \ref{aux.const.examp.}, $K^{\alpha}\widetilde{G}$ would be minimal, that is, $\widetilde{G}_{0}=\{0\}$, but this contradicts the choice we made for $\alpha$. As a consequence, $K^{\alpha}\widetilde{G}$ is not a simple algebra.

\bibliographystyle{abbrv}
\bibliography{Ob.reg}

\end{document}